\documentclass[a4paper,11pt]{amsart}

\usepackage[margin=1.2in,includehead]{geometry}
\usepackage[OT4]{fontenc}
\usepackage{amsmath,amsthm,amssymb}
\usepackage{enumitem}
\usepackage{graphicx}
\usepackage[pdftitle={New bounds on the maximum number of edges in k-quasi-planar graphs},
            pdfauthor={Andrew Suk and Bartosz Walczak},
            colorlinks=true,linkcolor=black,citecolor=black,filecolor=black,urlcolor=black]{hyperref}

\newtheorem{theorem}{Theorem}
\newtheorem{lemma}[theorem]{Lemma}
\newtheorem{corollary}[theorem]{Corollary}
\newtheorem{claim}[theorem]{Claim}

\renewenvironment{enumerate}{\begin{enumorig}[label=\textup{(\roman*)}, noitemsep, topsep=3pt plus 3pt, labelindent=1em, leftmargin=*, widest=]}{\end{enumorig}}

\renewenvironment{itemize}{\begin{itemorig}[label=\textbullet, noitemsep, topsep=3pt plus 3pt, labelindent=1em, leftmargin=*]}{\end{itemorig}}

\DeclareMathOperator{\up}{up}
\def\iter#1{\smash[t]{(}#1\smash[t]{)}}

\let\leq\leqslant
\let\geq\geqslant
\let\setminus\smallsetminus

\makeatletter
\let\old@setaddresses\@setaddresses
\def\@setaddresses{\bigskip\bgroup\parindent 0pt\let\scshape\relax\old@setaddresses\egroup}
\makeatother

\title[New bounds on the maximum number of edges in $k$-quasi-planar graphs]{\boldmath New bounds on the maximum number of edges in~$k$-quasi-planar graphs}

\author{Andrew Suk\and Bartosz Walczak}

\address[Andrew Suk]{University of Illinois at Chicago, Chicago, IL, USA}
\email{suk@uic.edu}

\address[Bartosz Walczak]{Theoretical Computer Science Department, Faculty of Mathematics and Computer Science, Jagiellonian University, Krak\'ow, Poland}
\email{walczak@tcs.uj.edu.pl}

\thanks{A journal version of this paper appeared in {\em Comput.\ Geom.}, 50:24--33, 2015.}
\thanks{A preliminary version of this paper appeared in Stephen Wismath and Alexander Wolff, editors, {\em Graph Drawing (GD 2013)}, volume 8242 of {\em Lecture Notes Comput.\ Sci.}, pages 95--106.\ Springer, Berlin, 2013.}

\thanks{Andrew Suk was supported by an NSF Postdoctoral Fellowship and by Swiss National Science Foundation grant 200021-125287/1.
Bartosz Walczak was supported by Swiss National Science Foundation grant 200020-144531 and by Ministry of Science and Higher Education of Poland grant 884/N-ESF-EuroGIGA/10/2011/0 within ESF EuroGIGA project GraDR\@.}

\begin{document}

\begin{abstract}
A topological graph is \emph{$k$-quasi-planar} if it does not contain $k$ pairwise crossing edges.  A 20-year-old conjecture asserts that for every fixed $k$, the maximum number of edges in a $k$-quasi-planar graph on $n$ vertices is $O(n)$.  Fox and Pach showed that every $k$-quasi-planar graph with $n$ vertices has at most $n(\log n)^{O(\log k)}$ edges.  We improve this upper bound to $2^{\alpha(n)^c}n\log n$, where $\alpha(n)$ denotes the inverse Ackermann function and $c$ depends only on $k$, for $k$-quasi-planar graphs in which any two edges intersect in a bounded number of points.  We also show that every $k$-quasi-planar graph with $n$ vertices in which any two edges have at most one point in common has at most $O(n\log n)$ edges.  This improves the previously known upper bound of $2^{\alpha(n)^c}n\log n$ obtained by Fox, Pach, and Suk.
\end{abstract}

\maketitle

\section{Introduction}

A \emph{topological graph} is a graph drawn in the plane so that its vertices are represented by points and its edges are represented by curves connecting the corresponding points.  The curves are always \emph{simple}, that is, they do not have self-intersections.  The curves are allowed to intersect each other, but they cannot pass through vertices except for their endpoints.  Furthermore, the edges are not allowed to have tangencies, that is, if two edges share an interior point, then they must properly cross at that point.  We only consider graphs without parallel edges or loops.  Two edges of a topological graph \emph{cross} if their interiors share a point.  A topological graph is \emph{simple} if any two of its edges have at most one point in common, which can be either a common endpoint or a crossing.

It follows from Euler's polyhedral formula that every topological graph on $n\geq 3$ vertices and with no two crossing edges has at most $3n-6$ edges.  A graph is called \emph{$k$-quasi-planar} if it can be drawn as a topological graph with no $k$ pairwise crossing edges.  Hence, a graph is $2$-quasi-planar if and only if it is planar.  According to a conjecture of Pach, Shahrokhi, and Szegedy \cite{PSS96} (see also \cite[Problem 1 in Section 9.6]{BMP-book}), for any fixed $k\geq 2$ there exists a constant $c_k$ such that every $k$-quasi-planar graph on $n$ vertices has at most $c_kn$ edges.  Agarwal, Aronov, Pach, Pollack, and Sharir \cite{AAP+97} were the first to prove this conjecture for \emph{simple} $3$-quasi-planar graphs.  Later, Pach, Radoi\v{c}i\'c, and T\'oth \cite{PRT06} generalized the result to \emph{all} $3$-quasi-planar graphs.  Ackerman \cite{Ack09} proved the conjecture for $k=4$.

For larger values of $k$, several authors have proved upper bounds on the maximum number of edges in $k$-quasi-planar graphs under various conditions on how the edges are drawn.  These include but are not limited to \cite{CaP92,FoP12,FPS13,PSS96,Val98}.  Fox and Pach \cite{FoP12} showed that every $k$-quasi-planar graph with $n$ vertices and no pair of edges intersecting in more than $t$ points has at most $n(c_t\frac{\log n}{\log k})^{c\log k}$ edges, where $c_t$ depends only on $t$ and $c$ is an absolute constant.  Recently, Fox and Pach \cite{FoP14} generalized this result, proving that every $k$-quasi-planar graph with $n$ vertices and without any restriction on the number of intersections between two edges has at most $n(\log n)^{c\log k}$ edges, where $c$ is an absolute constant.  In this paper, we improve the exponent of the polylogarithmic factor in the former bound from $O(\log k)$ to $1+o(1)$ for fixed~$t$.

\begin{theorem}
\label{tint}
Every\/ $k$-quasi-planar graph with\/ $n$ vertices and no pair of edges intersecting in more than\/ $t$ points has at most\/ $2^{\alpha(n)^c}n\log n$ edges, where\/ $\alpha(n)$ denotes the inverse of the Ackermann function, and\/ $c$ depends only on\/ $k$ and\/~$t$.
\end{theorem}

Recall that the \emph{Ackermann function} $A(n)$ is defined as follows.  Let $A_1(n)=2n$, and $A_k(n)=A_{k-1}(A_k(n-1))$ for $k\geq 2$.  In particular, we have $A_2(n)=2^n$, and $A_3(n)$ is an exponential tower of $n$ twos.  Now, let $A(n)=A_n(n)$, and let $\alpha(n)$ be defined as $\alpha(n)=\min\{k\geq 1\colon A(k)\geq n\}$.  This function grows much slower than the inverse of any primitive recursive function.

For \emph{simple} topological graphs, Fox, Pach, and Suk \cite{FPS13} showed that every $k$-quasi-planar simple topological graph on $n$ vertices has at most $2^{\alpha(n)^c}n\log n$ edges, where $c$ depends only on~$k$.  We establish the following improvement.

\begin{theorem}
\label{simple}
Every\/ $k$-quasi-planar simple topological graph on\/ $n$ vertices has at most\/ $c_kn\log n$ edges, where\/ $c_k$ depends only on\/~$k$.
\end{theorem}

We start the proofs of both theorems with a reduction to the case of topological graphs containing an edge that intersects every other edge.  This reduction introduces the $O(\log n)$ factor for the bound on the number of edges.  Then, the proof of Theorem \ref{tint} follows the approaches of Valtr \cite{Val98} and Fox, Pach, and Suk \cite{FPS13}, using a result on generalized Davenport-Schinzel sequences, which we recall in Section~\ref{ds}.  Although the proofs in \cite{Val98} and \cite{FPS13} heavily depend on the assumption that any two edges have at most one point in common, we are able to remove this condition by establishing some technical lemmas in Section~\ref{intpat}.  In Section \ref{pftint}, we finish the proof of Theorem~\ref{tint}.  The proof of Theorem \ref{simple}, which relies on a recent coloring result due to Laso\'n, Micek, Pawlik, and Walczak \cite{LMPW14}, is given in Section~\ref{bart}.

\section{Initial reduction}

We call a collection $C$ of curves in the plane \emph{decomposable} if there is a partition $C=C_1\cup\cdots\cup C_w$ such that each $C_i$ contains a curve intersecting all other curves in $C_i$, and for $i\neq j$, no curve in $C_i$ crosses nor shares an endpoint with a curve in~$C_j$.

\begin{lemma}[Fox, Pach, Suk {\cite[Lemma 3.2]{FPS13}}]
\label{decompose}
There is an absolute constant\/ $c>0$ such that every collection\/ $C$ of\/ $m\geq 2$ curves such that any two of them intersect in at most\/ $t$ points has a decomposable subcollection of size at least\/ $\frac{cm}{t\log m}$.
\end{lemma}

In the proofs of both Theorem \ref{tint} and Theorem \ref{simple}, we establish a (near) linear upper bound on the number of edges under the additional assumption that the graph has an edge intersecting every other edge.  Once this is achieved, we use the following lemma to infer an upper bound for the general case.

\begin{lemma}[implicit in \cite{FPS13}]
\label{reduction}
Let\/ $G$ be a topological graph on\/ $n$ vertices such that no two edges have more than\/ $t$ points in common.  Suppose that for some constant\/ $\beta$, every subgraph\/ $G'$ of\/ $G$ containing an edge that intersects every other edge of\/ $G'$ has at most\/ $\beta|V(G')|$ edges.  Then\/ $G$ has at most\/ $c_t\beta n\log n$ edges, where\/ $c_t$ depends only on\/~$t$.
\end{lemma}

\begin{proof}
By Lemma \ref{decompose}, there is a decomposable subset $E'\subset E(G)$ such that $|E'|\geq c'_t|E(G)|/\log|E(G)|$, where $c'_t$ depends only on~$t$.  Hence, there is a partition $E'=E_1\cup\cdots\cup E_w$, such that each $E_i$ has an edge $e_i$ that intersects every other edge in $E_i$, and for $i\neq j$, the edges in $E_i$ are disjoint from the edges in~$E_j$.  Let $V_i$ denote the set of vertices that are the endpoints of the edges in $E_i$, and let $n_i=|V_i|$.  By the assumption, we have $|E_i|\leq\beta n_i$ for $1\leq i\leq w$.  Hence,
\begin{equation*}
\frac{c'_t|E(G)|}{\log|E(G)|}\leq|E'|\leq\sum_{i=1}^w\beta n_i\leq\beta n.
\end{equation*}
Since $|E(G)|\leq n^2$, we obtain $|E(G)|\leq 2(c'_t)^{-1}\beta n\log n$.
\end{proof}

\section{Generalized Davenport-Schinzel sequences}\label{ds}

A sequence $S=(s_1,\ldots,s_m)$ is called \emph{$l$-regular} if any $l$ consecutive terms of $S$ are pairwise different.  For integers $l,m\geq 2$, the sequence $S=(s_1,\ldots,s_{lm})$ is said to be of \emph{type\/ $\up(l,m)$} if the first $l$ terms are pairwise different and $s_i=s_{i+l}=\cdots=s_{i+(m-1)l}$ for $1\leq i\leq l$.  In particular, every sequence of type $\up(l,m)$ is $l$-regular.  For convenience, we will index the elements of an $\up(l,m)$ sequence as
\begin{equation*}
S=(s_{1,1},\ldots,s_{l,1},\;s_{1,2},\ldots,s_{l,2},\;\ldots,\;s_{1,m},\ldots,s_{l,m}),
\end{equation*}
where $s_{1,1},\ldots,s_{l,1}$ are pairwise different and $s_{i,1}=\cdots=s_{i,m}$ for $1\leq i\leq l$.

\begin{theorem}[\cite{Kla92}, see also (18) in \cite{Kla02}]
\label{klazar}
For\/ $l\geq 2$ and\/ $m\geq 3$, every\/ $l$-regular sequence over an\/ $n$-element alphabet that does not contain a subsequence of type\/ $\up(l,m)$ has length at most
\begin{equation*}
n\cdot l\cdot 2^{(lm-3)}\cdot(10l)^{10\alpha(n)^{lm}}.
\end{equation*}
\end{theorem}

\noindent For more results on generalized Davenport-Schinzel sequences, see \cite{Niv10,Pet11a,Pet11b}.

\section{Intersection pattern of curves}\label{intpat}

In this section, we will prove several technical lemmas on the intersection pattern of curves in the plane.  We will always assume that no two curves are tangent, and that if two curves share an interior point, then they must properly cross at that point.

\begin{lemma}
\label{disjoint}
Let\/ $\lambda_1$ and\/ $\lambda_2$ be disjoint simple closed curves.  Let\/ $C$ be a collection of\/ $m$ curves with one endpoint on\/ $\lambda_1$, the other endpoint on\/ $\lambda_2$, and no other common points with\/ $\lambda_1$ or\/~$\lambda_2$.  If no\/ $k$ members of\/ $C$ pairwise cross, then\/ $C$ contains\/ $\lceil m/(k-1)^2\rceil$ pairwise disjoint members.
\end{lemma}

\begin{proof}
Let $G$ be the intersection graph of~$C$.  Since $G$ does not contain a clique of size $k$, by Tur\'an's theorem \cite{Tur41}, $|E(G)|\leq(1-1/(k-1))m^2/2$.  Hence, there is a curve $a\in C$ and a subset $S\subset C$, such that $|S|\geq m/(k-1) -1$ and $a$ is disjoint from every curve in~$S$.  We order the elements in $S\cup\{a\}$ as $a_0,a_1,\ldots,a_{|S|}$ in clockwise order as their endpoints appear on $\lambda_1$, starting with $a_0=a$.  Now, we define the partial order $\prec$ on the pairs in $S$ so that $a_i\prec a_j$ if $i<j$ and $a_i$ is disjoint from~$a_j$.  A simple geometric observation shows that $\prec$ is indeed a partial order.  Since $S$ does not contain $k$ pairwise crossing members, by Dilworth's theorem \cite{Dil50}, $S\cup\{a\}$ contains $\lceil m/(k-1)^2\rceil$ pairwise disjoint members.
\end{proof}

A collection of curves with a common endpoint $v$ is called a \emph{fan} with \emph{apex}~$v$.  Let $C=\{a_1,\ldots,a_m\}$ be a fan with apex $v$, and $\gamma=\gamma_1\cup\cdots\cup\gamma_m$ be a curve with endpoints $p$ and $q$ partitioned into $m$ subcurves $\gamma_1,\ldots,\gamma_m$ that appear in this order along $\gamma$ from $p$ to~$q$.  We say that $C$ is \emph{grounded} by $\gamma_1\cup\cdots\cup\gamma_m$ if
\begin{enumerate}
\item $\gamma$ does not contain~$v$,
\item each $a_i$ has its other endpoint on~$\gamma_i$.
\end{enumerate}
We say that $C$ is \emph{well-grounded} by $\gamma_1\cup\cdots\cup\gamma_m$ if $C$ is grounded by $\gamma_1\cup\cdots\cup\gamma_m$ and each $a_i$ intersects $\gamma$ only within~$\gamma_i$.  Note that both notions depend on a particular partition $\gamma=\gamma_1\cup\cdots\cup\gamma_m$.  See Figure \ref{wellgroundfig} for a small example.

\begin{figure}[t]
\begin{center}
\includegraphics[width=150pt]{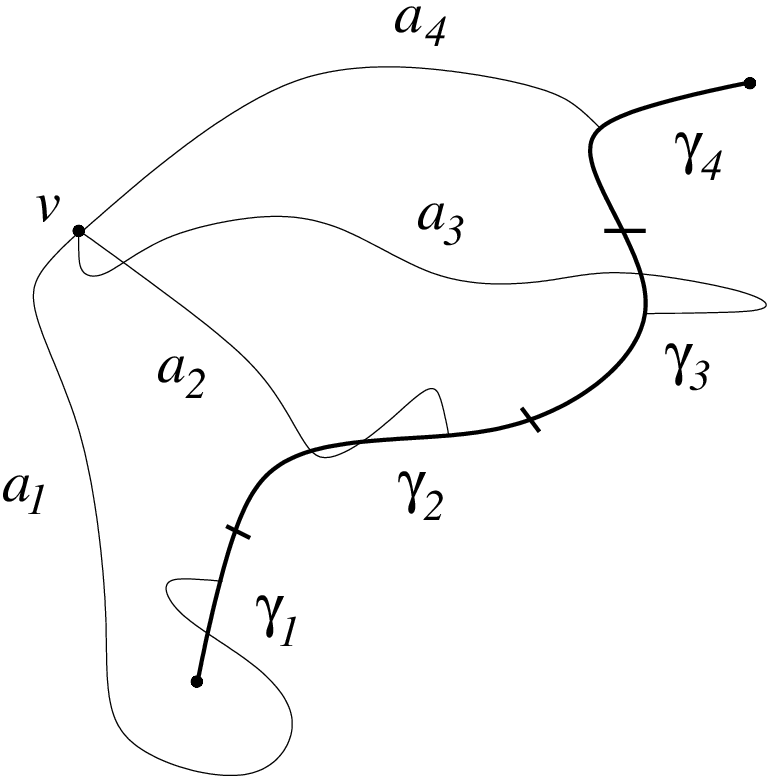}
\caption{Fan $C=\{a_1,a_2,a_3,a_4\}$ is well-grounded by $\gamma=\gamma_1\cup\gamma_2\cup\gamma_3\cup\gamma_4$.}
\label{wellgroundfig}
\end{center}
\end{figure}

\begin{lemma}
\label{twostep}
Let\/ $C=\{a_1,\ldots,a_m\}$ be a fan grounded by a curve\/ $\gamma=\gamma_1\cup\cdots\cup\gamma_m$.  If each\/ $a_i$ intersects\/ $\gamma$ in at most\/ $t$ points, then there is a subfan\/ $C'=\{a_{i_1},\ldots,a_{i_r}\}\subset C$ with\/ $i_1<\cdots<i_r$ and\/ $r=\lfloor\log_{t+1}m\rfloor$ that is grounded by a subcurve\/ $\gamma'=\gamma'_1\cup\cdots\cup\gamma'_r\subset\gamma$.  Moreover,
\begin{enumerate}
\item $\gamma'_j\supset\gamma_{i_j}$ for\/ $1\leq j\leq r$,
\item $a_{i_j}$ intersects\/ $\gamma'$ only within\/ $\gamma'_1\cup\cdots\cup\gamma'_j$ for\/ $1\leq j\leq r$.
\end{enumerate}
\end{lemma}

\begin{proof}
We proceed by induction on~$m$.  The base case $m\leq t$ is trivial.  Now, assume that $m\geq t+1$ and the statement holds up to $m-1$.  Since $a_1$ intersects $\gamma$ in at most $t$ points, there exists an integer $j$ such that $a_1$ is disjoint from $\gamma_j\cup\gamma_{j+1}\cup\cdots\cup\gamma_{j+\lfloor m/(t+1)\rfloor-1}$.  By the induction hypothesis applied to $\{a_j,a_{j+1},\ldots,a_{j+\lfloor m/(t+1)\rfloor-1}\}$ and the curve $\gamma_j\cup\gamma_{j+1}\cup\cdots\cup\gamma_{j+\lfloor m/(t+1)\rfloor-1}$, we obtain a subfan $C^{\ast}=\{a_{i_2},\ldots,a_{i_r}\}$ of $r-1=\lfloor\log_{t+1}\lfloor m/(t+1)\rfloor\rfloor=\lfloor\log_{t+1}m\rfloor-1$ curves, and a subcurve $\gamma^{\ast}=\gamma'_2\cup\cdots\cup\gamma'_r\subset\gamma_j\cup\gamma_{j+1}\cup\cdots\cup\gamma_{j+\lfloor m/(t+1)\rfloor-1}$ with the desired properties.  Let $\gamma'_1$ be the subcurve of $\gamma$ obtained by extending the endpoint of $\gamma_1$ to the endpoint of $\gamma^{\ast}$ along $\gamma$ so that $\gamma'_1\supset\gamma_1$.  Set $\gamma'=\gamma'_1\cup\gamma^{\ast}$.  Hence, the collection of curves $C'=\{a_1\}\cup C^{\ast}$ and $\gamma'$ have the desired properties.
\end{proof}

\begin{lemma}
\label{fan}
Let\/ $C=\{a_1,\ldots,a_m\}$ be a fan grounded by a curve\/ $\gamma=\gamma_1\cup\cdots\cup\gamma_m$.  If each\/ $a_i$ intersects\/ $\gamma$ in at most\/ $t$ points, then there is a subfan\/ $C'=\{a_{i_1},\ldots,a_{i_r}\}\subset C$ with\/ $i_1<\cdots<i_r$ and\/ $r=\lfloor\log_{t+1}\log_{t+1}m\rfloor$ that is well-grounded by a subcurve\/ $\gamma'=\gamma_1'\cup\cdots\cup\gamma'_r\subset\gamma$.  Moreover, $\gamma'_j\supset\gamma_{i_j}$ for\/ $1\leq j\leq r$.
\end{lemma}

\begin{proof}
We apply Lemma \ref{twostep} to $C$ and $\gamma=\gamma_1\cup\cdots\cup\gamma_m$ to obtain a subcollection $C^{\ast}=\{a_{j_1},a_{j_2},\ldots,a_{j_{m^{\ast}}}\}$ of $m^{\ast}=\lfloor\log_{t+1}m\rfloor$ curves, and a subcurve $\gamma^{\ast}=\gamma^{\ast}_1\cup\cdots\cup\gamma^{\ast}_{m^{\ast}}\subset\gamma$ with the properties listed in Lemma~\ref{twostep}.  Then we apply Lemma \ref{twostep} again to $C^{\ast}$ and $\gamma^{\ast}$ with the elements in $C^{\ast}$ in reverse order.  By the second property of Lemma \ref{twostep}, the resulting subcollection $C'=\{a_{i_1},\ldots,a_{i_r}\}$ of $r=\lfloor\log_{t+1}\log_{t+1}m\rfloor$ curves is well-grounded by a subcurve $\gamma'=\gamma'_1\cup\cdots\cup\gamma'_r\subset\gamma$, and by the first property we have $\gamma'_j\supset\gamma_{i_j}$ for $1\leq j\leq r$.
\end{proof}

We say that fans $C_1,\ldots,C_l$ are \emph{simultaneously grounded} (\emph{simultaneously well-grounded}) by a curve $\gamma=\gamma_1\cup\cdots\cup\gamma_m$ to emphasize that they are grounded (well-grounded) by $\gamma$ with \emph{the same} partition $\gamma=\gamma_1\cup\cdots\cup\gamma_m$.  See Figure \ref{simwellground} for a small example.

\begin{figure}[t]
\begin{center}
\includegraphics[width=150pt]{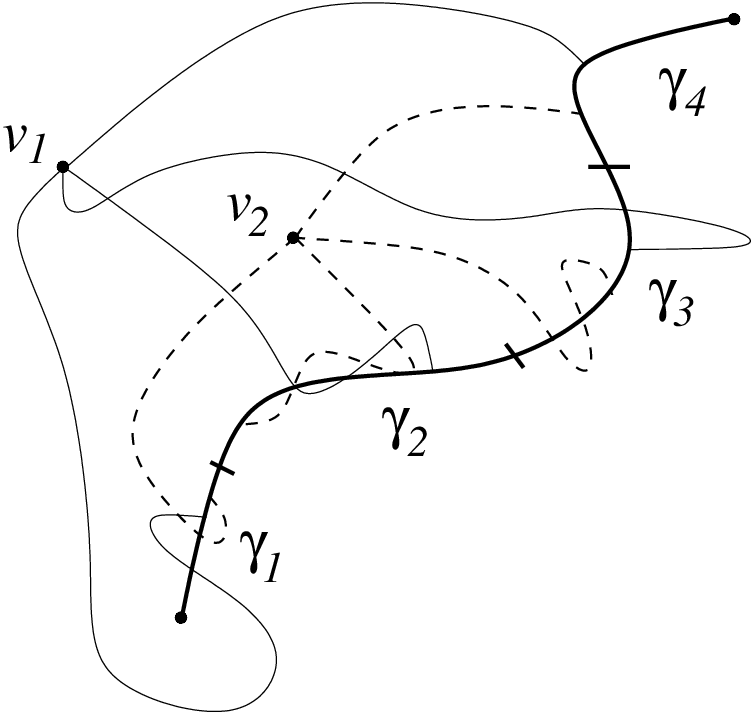}
\caption{Fans $C_1$ and $C_2$ are simultaneously well-grounded by $\gamma=\gamma_1\cup\gamma_2\cup\gamma_3\cup\gamma_4$.}
\label{simwellground}
\end{center}
\end{figure}

\begin{lemma}
\label{grounded}
Let\/ $C_1,\ldots,C_l$ be\/ $l$ fans with\/ $C_i=\{a_{i,1},\ldots,a_{i,m}\}$ that are simultaneously grounded by a curve\/ $\gamma=\gamma_1\cup\cdots\cup\gamma_m$.  If each\/ $a_{i,j}$ intersects\/ $\gamma$ in at most\/ $t$ points, then there are indices\/ $j_1<\cdots<j_r$ with\/ $r=\lfloor\log_{t+1}^{\iter{2l}}m\rfloor$ (\/$2l$-times iterated logarithm of\/ $m$) and a subcurve\/ $\gamma'=\gamma'_1\cup\cdots\cup\gamma'_r\subset\gamma$ such that
\begin{enumerate}
\item the subfans\/ $C'_i=\{a_{i,j_1},\ldots,a_{i,j_r}\}\subset C_i$ for\/ $1\leq i\leq l$ are simultaneously well-grounded by\/ $\gamma'_1\cup\cdots\cup\gamma'_r$,
\item $\gamma'_s\supset\gamma_{j_s}$ for\/ $1\leq s\leq r$.
\end{enumerate}
\end{lemma}

\begin{proof}
We proceed by induction on~$l$.  The base case $l=1$ follows from Lemma~\ref{fan}.  Now, assume the statement holds up to $l-1$.  We apply Lemma \ref{fan} to the fan $C_1=\{a_{1,1},\ldots,a_{1,m}\}$ and the curve $\gamma=\gamma_1\cup\cdots\cup\gamma_m$, to obtain a subfan $C^{\ast}_1=\{a_{1,w_1},\ldots,a_{1,w_s}\}\subset C_1$ with $w_1<\cdots<w_s$ and $s=\lfloor\log_{t+1}\log_{t+1}m\rfloor$ that is well-grounded by a subcurve $\gamma^{\ast}=\gamma^{\ast}_1\cup\cdots\cup\gamma^{\ast}_s\subset\gamma$ and satisfies $\gamma^{\ast}_i\supset\gamma_{w_i}$ for $1\leq i\leq s$.  For $2\leq i\leq l$, let $C^{\ast}_i=\{a_{i,w_1},\ldots,a_{i,w_s}\}\subset C_i$.  Now, we apply the induction hypothesis on the collection of $l-1$ fans $C^{\ast}_2,\ldots,C^{\ast}_l$ that are simultaneously grounded by the curve $\gamma^{\ast}=\gamma^{\ast}_1\cup\cdots\cup\gamma^{\ast}_s$.  Hence, we obtain indices $j_1<\cdots<j_r$ with $r=\lfloor\log_{t+1}^{\iter{2l-2}}s\rfloor=\lfloor\log_{t+1}^{\iter{2l}}m\rfloor$ and a subcurve $\gamma'=\gamma'_1\cup\cdots\cup\gamma'_r\subset\gamma^{\ast}$ such that each subfan $C'_i=\{a_{i,j_1},\ldots,a_{i,j_r}\}\subset C_i$ with $2\leq i\leq l$ is well-grounded by $\gamma'_1\cup\cdots\cup\gamma'_r$, and moreover $\gamma'_z\supset\gamma^{\ast}_z\supset\gamma_z$ for $1\leq z\leq r$.  By setting $C'_1=\{a_{1,j_1},\ldots,a_{1,j_r}\}\subset C^{\ast}_1$, the fans $C'_1,\ldots,C'_l$ are simultaneously well-grounded by the subcurve $\gamma'=\gamma'_1\cup\cdots\cup\gamma'_r\subset\gamma$.
\end{proof}

Let $C=\{a_1,\ldots,a_m\}$ be a fan with apex $v$ grounded by a curve $\gamma=\gamma_1\cup\cdots\cup\gamma_m$ with endpoints $p$ and~$q$.  We say that $a_i$ is \emph{left-sided} (\emph{right-sided}) if moving along $a_i$ from $v$ until we reach $\gamma$ for the first time, and then turning left (right) onto the curve $\gamma$, we reach the endpoint~$q$.  We say that $C_i$ is \emph{one-sided}, if the curves in $C_i$ are either all left-sided or all right-sided.

\begin{lemma}
\label{finlem}
Let\/ $C_1,\ldots,C_l$ be\/ $l$ fans with\/ $C_i=\{a_{i,1},\ldots,a_{i,m}\}$ that are simultaneously grounded by a curve\/~$\gamma$.  Then there are indices\/ $j_1<\cdots<j_r$ with\/ $r=\lceil m/2^l\rceil$ such that the subfans\/ $C'_i=\{a_{i,j_1},\ldots,a_{i,j_r}\}\subset C_i$ for\/ $1\leq i\leq l$ are one-sided.
\end{lemma}

\begin{proof}
We proceed by induction on~$l$.  The base case $l=1$ is trivial since at least half of the curves in $C_1=\{a_{1,1},\ldots,a_{1,m}\}$ form a one-sided subset.  For the inductive step, assume that the statement holds up to $l-1$.  Let $C^{\ast}_1=\{a_{1,w_1},\ldots,a_{1,w_{\lceil m/2\rceil}}\}$ with $w_1<\cdots<w_{\lceil m/2\rceil}$ be a subset of $\lceil m/2\rceil$ curves that is one-sided.  For $i\geq 2$, set $C^{\ast}_i=\{a_{i,w_1},\ldots,a_{i,w_{\lceil m/2\rceil}}\}$.  Then apply the induction hypothesis on the $l-1$ fans $C^{\ast}_2,\ldots,C^{\ast}_l$, to obtain indices $j_1<\cdots<j_r$ with $r=\lceil\lceil m/2\rceil/2^{l-1}\rceil=\lceil m/2^l\rceil$ such that the subfans $C'_i=\{a_{i,j_1},\ldots,a_{i,j_r}\}\subset C^{\ast}_i$ for $2\leq i\leq l$ are one-sided.  By setting $C'_1=\{a_{1,j_1},\ldots,a_{1,j_r}\}\subset C^{\ast}_1$, the subfans $C'_1,\ldots,C'_l$ have the desired properties.
\end{proof}

Since at least half of the fans obtained from Lemma \ref{finlem} are either left-sided or right-sided, we have the following corollary.

\begin{corollary}
\label{leftright}
Let\/ $C_1,\ldots,C_{2l}$ be\/ $2l$ fans with\/ $C_i=\{a_{i,1},\ldots,a_{i,m}\}$ that are simultaneously grounded by a curve\/~$\gamma$.  Then there are indices\/ $i_1<\cdots<i_l$ and\/ $j_1<\cdots<j_r$ with\/ $r=\lceil m/2^{2l}\rceil$ such that the subfans\/ $C'_{i_w}=\{a_{i_w,j_1},\ldots,a_{i_w,j_r}\}\subset C_{i_w}$ for\/ $1\leq w\leq l$ are all left-sided or all right-sided.
\end{corollary}

By combining Lemma \ref{grounded} and Corollary \ref{leftright}, we easily obtain the following lemma which will be used in Section~\ref{pftint}.

\begin{lemma}
\label{use}
Let\/ $C_1,\ldots,C_{2l}$ be\/ $2l$ fans with\/ $C_i=\{a_{i,1},\ldots,a_{i,m}\}$ that are simultaneously grounded by a curve\/ $\gamma=\gamma_1\cup\cdots\cup\gamma_m$.  If each\/ $a_{i,j}$ intersects\/ $\gamma$ in at most\/ $t$ points, then there are indices\/ $i_1<\cdots<i_l$ and\/ $j_1<\cdots<j_r$ with\/ $r=\lceil\lfloor\log_{t+1}^{\iter{4l}}m\rfloor/2^{2l}\rceil$ and a subcurve\/ $\gamma^{\ast}=\gamma^{\ast}_1\cup\cdots\cup\gamma^{\ast}_r\subset\gamma$ such that
\begin{enumerate}
\item the subfans\/ $C'_{i_w}=\{a_{i_w,j_1},\ldots,a_{i_w,j_r}\}\subset C_{i_w}$ for\/ $1\leq w\leq l$ are simultaneously well-grounded by\/ $\gamma^{\ast}_1\cup\cdots\cup\gamma^{\ast}_r$,
\item $\gamma^{\ast}_s\supset\gamma_{j_s}$ for\/ $1\leq s\leq r$,
\item the subfans\/ $C'_{i_1},\ldots,C'_{i_l}$ are all left-sided or all right-sided.
\end{enumerate}
\end{lemma}

\section{Proof of Theorem \ref{tint}}\label{pftint}

By Lemma \ref{reduction} and the fact that the function $\alpha(n)$ is non-decreasing, it is enough to prove that every $k$-quasi-planar topological graph on $n$ vertices such that
\begin{itemize}
\item no two edges have more than $t$ points in common,
\item there is an edge that intersects every other edge,
\end{itemize}
has at most $2^{\alpha(n)^c}n$ edges, where $c$ depends only on $k$ and~$t$.

Let $G$ be a $k$-quasi-planar graph on $n$ vertices with no two edges intersecting in more than $t$ points.  Let $e_0=pq$ be an edge that intersects every other edge of~$G$.  Let $V_0=V(G)\setminus\{p,q\}$ and $E_0$ be the set of edges with both endpoints in~$V_0$.  Hence, we have $|E_0|>|E(G)|-2n$.  Assume without loss of generality that no two elements of $E_0$ cross $e_0$ at the same point.

It is a well-known fact (see e.g.\ \cite[Theorem 2.2.1]{AlS-book}) that there is a bipartition $V_0=V_1\cup V_2$ such that at least half of the edges in $E_0$ connect a vertex in $V_1$ to a vertex in~$V_2$.  Let $E_1$ be the set of these edges.  For each vertex $v_i\in V_1$, consider the graph $G_i$ whose each vertex corresponds to the subcurve $\gamma$ of an edge $e\in E_1$ such that
\begin{enumerate}
\item $e$ is incident to~$v_i$,
\item the endpoints of $\gamma\subset e$ are $v_i$ and the first intersection point in $e\cap e_0$ as moving from $v_i$ along~$e$.
\end{enumerate}
Two vertices are adjacent in $G_i$ if the corresponding subcurves cross.  Each graph $G_i$ is isomorphic to the intersection graph of a collection of curves with one endpoint on a simple closed curve $\lambda_1$ and the other endpoint on a simple closed curve $\lambda_2$ and with no other points in common with $\lambda_1$ or~$\lambda_2$.  To see this, enlarge the point $v_i$ and the curve $e_0$ a little, making them simple closed curves $\lambda_1$ and $\lambda_2$, and shorten the curves $\gamma$ appropriately, so as to preserve all crossings between them.  Since no $k$ of these curves pairwise intersect, by Lemma \ref{disjoint}, $G_i$ contains an independent set of size $\lceil|V(G_i)|/(k-1)^2\rceil$.  We keep all edges corresponding to the elements of this independent set, and discard all other edges incident to~$v_i$.  After repeating this process for all vertices in $V_1$, we are left with at least $\lceil|E_1|/(k-1)^2\rceil$ edges, forming a set~$E_2$.  We continue this process on the vertices in $V_2$ and the edges in~$E_2$.  After repeating this process for all vertices in $V_2$, we are left with at least $\lceil|E_2|/(k-1)^2\rceil$ edges, forming a set~$E'$.  Thus $|E(G)|<2(k-1)^4|E'|+2n$.  Now, for any two edges $e_1,e_2\in E'$ that share an endpoint, the subcurves $\gamma_1\subset e_1$ and $\gamma_2\subset e_2$ described above must be disjoint.

For each edge $e\in E'$, fix an arbitrary intersection point $s\in e\cap e_0$ to be the \emph{main intersection point} of $e$ and~$e_0$.  Let $e_1,\ldots,e_{|E'|}$ denote the edges in $E'$ listed in the order their main intersection points appear on $e_0$ from $p$ to $q$, and let $s_1,\ldots,s_{|E'|}$ denote these points respectively.  We label the endpoints of each $e_i$ as $p_i$ and $q_i$, as follows.  As we move along $e_0$ from $p$ to $q$ until we arrive at $s_i$, then we turn left and move along $e_i$, we finally reach $p_i$, while as we turn right at $s_i$ and move along $e_i$, we finally reach~$q_i$.  We define sequences $S_1=(p_1,\ldots,p_{|E'|})$ and $S_2=(q_1,\ldots,q_{|E'|})$.  They are sequences of length $|E'|$ over the $(n-2)$-element alphabet~$V_0$.  See Figure \ref{s1s2new} for a small example.

\begin{figure}[t]
\begin{center}
\includegraphics[width=170pt]{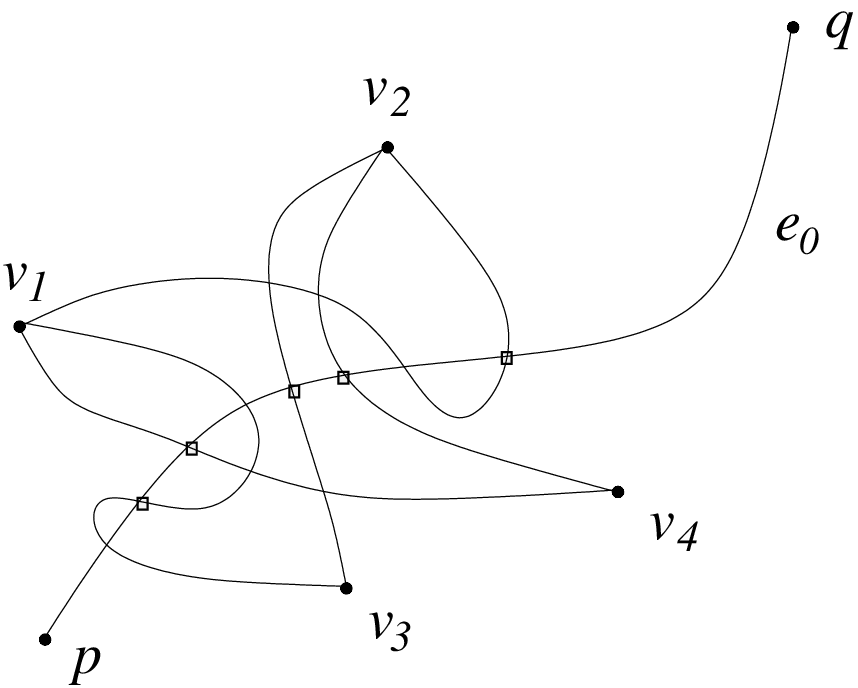}
\caption{In this example, main intersection points are indicated by squares and we have $S_1=(v_3,v_1,v_2,v_2,v_2)$ and $S_2=(v_1,v_4,v_3,v_4,v_1)$.}
\label{s1s2new}
\end{center}
\end{figure}

We will use the following lemma, due to Valtr \cite{Val98}, to find a large subsequence in either $S_1$ or $S_2$ that is $2l$-regular.  We include the proof for completeness.

\begin{lemma}[Valtr {\cite[Lemma 5]{Val98}}]
\label{regular}
For\/ $2l\geq 1$, at least one of the sequences\/ $S_1,S_2$ defined above contains a\/ $2l$-regular subsequence of length at least\/ $\lceil|E'|/(8l)\rceil$.
\end{lemma}

\begin{proof}
Given an integer $2l$ and a sequence $S$ (of vertices), we apply a greedy algorithm that returns a $2l$-regular subsequence $R(S,2l)$.  At the beginning of the algorithm, an auxiliary sequence $R$ is taken empty.  Then, the terms of $S$ are considered one by one from left to right, and at each step the considered term from $S$ is placed at the right end of $R$ if it does not violate the $2l$-regularity of~$R$.  Otherwise, the algorithm continues to the next term in~$S$.  Once all terms are considered in $S$, the algorithm terminates and returns a $2l$-regular subsequence $R(S,2l)=R$.  We let $|S|$ denote the length of a sequence $S$, and let $I(S)$ denote the set of vertices occurring in~$S$.

Recall that $S_1=(p_1,\ldots,p_{|E'|})$ and $S_2=(q_1,\ldots,q_{|E'|})$.  Given integers $j_1,j_2$ such that $1\leq j_1\leq j_2\leq|E'|$, we let $S_{1,[j_1,j_2]}=(p_{j_1},\ldots,p_{j_2})$ and $S_{2,[j_1,j_2]}=(q_{j_1},\ldots,q_{j_2})$.
We have
\begin{equation*}
\{e_{j_1},\ldots,e_{j_2}\} \subset \{(p_a,q_b)\colon p_a \in I(S_{1,[j_1,j_2]}),\: q_b \in I(S_{2,[j_1,j_2]})\},
\end{equation*}
which implies
\begin{equation*}
|\{e_{j_1},\ldots,e_{j_2}\}| \leq |\{(p_a,q_b)\colon p_a \in I(S_{1,[j_1,j_2]}),\: q_b \in I(S_{2,[j_1,j_2]})\}|.
\end{equation*}
Therefore,
\begin{equation*}
j_2-j_1+1 \leq |I(S_{1,[j_1,j_2]})| \cdot |I(S_{2,[j_1,j_2]})|.
\end{equation*}
By the inequality of arithmetic and geometric means, we have
\begin{equation}
\frac{|I(S_{1,[j_1,j_2]})| + |I(S_{2,[j_1,j_2]})|}{2} \geq \sqrt{j_2-j_1+1}.
\tag{$*$}\label{keyeq}
\end{equation}

We are going to prove that for each $j$ with $1\leq j\leq|E'|$, we have
\begin{equation}
|R(S_{1,[1,j]},2l)| + |R(S_{2,[1,j]},2l)| \geq \frac{j}{4l}.
\tag{$**$}\label{lemc}
\end{equation}
We proceed by induction on~$j$.  For the base cases $j\leq\min\{64l^2,|E'|\}$, by \eqref{keyeq} and $j\leq 64l^2$, we have
\begin{equation*}
|R(S_{1,[1,j]},2l)| + |R(S_{2,[1,j]},2l)| \geq |I(S_{1,[1,j]})| + |I(S_{2,[1,j]})| \geq 2\sqrt{j} \geq j/(4l).
\end{equation*}
Now, suppose that $64l^2<j_0\leq|E'|$ and \eqref{lemc} holds for $1\leq j\leq j_0-1$.  Note that for $i\in\{1,2\}$, each vertex in $I(S_{i,[j_0-16l^2+1,j_0]})$ not occurring among the last $2l-1$ terms of $R=R(S_{i,[1,j_0-16l^2]},2l)$ will eventually be added to $R$ by the greedy algorithm.  Therefore,
\begin{equation*}
|R(S_{i,[1,j_0]},2l)| \geq |R(S_{i,[1,j_0-16l^2]},2l)| + |I(S_{i,[j_0-16l^2+1,j_0]})| - (2l-1).
\end{equation*}
By the induction hypothesis and by \eqref{keyeq}, we have
\begin{equation*}
|R(S_{1,[1,j_0]},2l)| + |R(S_{2,[1,j_0]},2l)| \geq (j_0-16l^2)/(4l) + 2\sqrt{16l^2} - 2(2l-1) \geq j_0/(4l).
\end{equation*}
This completes the proof of \eqref{lemc}.  Now, Lemma \ref{regular} follows from \eqref{lemc} with $j=|E'|$ and from the pigeonhole principle.
\end{proof}

For the rest of this section, we set $l=2^{k^2+2k}$ and $m$ to be such that $(\log_{t+1}^{\iter{4l}}m)/2^{2l}=3\cdot 2^k-4$.

\begin{lemma}
\label{key}
Neither of the sequences\/ $S_1$ and\/ $S_2$ has a subsequence of type\/ $\up(2l,m)$.
\end{lemma}

\begin{proof}
By symmetry, it suffices to show that $S_1$ does not contain a subsequence of type $\up(2l,m)$.  We will prove that the existence of such a subsequence would imply that $G$ has $k$ pairwise crossing edges.  Let
\begin{equation*}
S=(s_{1,1},\ldots,s_{2l,1},\;s_{1,2},\ldots,s_{2l,2},\;\ldots,\;s_{1,m},\ldots,s_{2l,m})
\end{equation*}
be a subsequence of $S_1$ of type $\up(2l,m)$ such that the first $2l$ terms are pairwise distinct and $s_{i,1}=\cdots=s_{i,m}=v_i$ for $1\leq i\leq 2l$.  For $1\leq j\leq m$, let $a_{i,j}$ be the subcurve of the edge corresponding to the entry $s_{i,j}$ in $S_1$ between the vertex $v_i$ and the main intersection point with~$e_0$.  Let $C_i=\{a_{i,1},\ldots,a_{i,m}\}$ for $1\leq i\leq 2l$.  Hence, $C_1,\ldots,C_{2l}$ are $2l$ fans with apices $v_1,\ldots,v_{2l}$ respectively.  Clearly, there is a partition $e_0=\gamma_1\cup\cdots\cup\gamma_m$ such that $C_1,\ldots,C_{2l}$ are simultaneously grounded by $\gamma_1\cup\cdots\cup\gamma_m$.

We apply Lemma \ref{use} to the fans $C_1,\ldots,C_{2l}$ that are simultaneously grounded by $\gamma_1\cup\cdots\cup\gamma_m$ to obtain indices $i_1<\cdots<i_l$ and $j_1<\cdots<j_r$ with $r=(\log_{t+1}^{\iter{4l}}m)/2^{2l}=3\cdot 2^k-4$ and a subcurve $\gamma^{\ast}=\gamma^{\ast}_1\cup\cdots\cup\gamma^{\ast}_r\subset e_0$ such that
\begin{enumerate}
\item the subfans $C'_{i_w}=\{a_{i_w,j_1},\ldots,a_{i_w,j_r}\}\subset C_{i_w}$ for $1\leq w\leq l$ are simultaneously well-grounded by $\gamma^{\ast}_1\cup\cdots\cup\gamma^{\ast}_r$,
\item $\gamma^{\ast}_z\supset\gamma_{j_z}$ for $1\leq z\leq r$,
\item the subfans $C'_{i_1},\ldots,C'_{i_l}$ are all left-sided or all right-sided.
\end{enumerate}
We will only consider the case that $C'_{i_1},\ldots,C'_{i_l}$ are left-sided, the other case being symmetric.

Now, for $1\leq w\leq l$ and $1\leq z\leq r$, we define the subcurve $a^{\ast}_{w,z}\subset a_{i_w,j_z}$ whose endpoints are $v_{i_w}$ and the first point from $a_{i_w,j_z}\cap\gamma^{\ast}$ as moving from $v_{i_w}$ along $a_{i_w,j_z}$.  Hence, the interior of $a^{\ast}_{w,z}$ is disjoint from $\gamma^{\ast}$.  Let $A^{\ast}_w=\{a^{\ast}_{w,1},\ldots,a^{\ast}_{w,r}\}$ for $1\leq w\leq l$.  Note that any two curves in $A^{\ast}_w$ do not cross by construction, and all curves in $A^{\ast}_w$ enter $\gamma^{\ast}$ from the same side.  For simplicity, we will call this the \emph{left side} of $\gamma^{\ast}$ and we will relabel the apices of the fans $A^{\ast}_1,\ldots,A^{\ast}_l$ from $v_{i_1},\ldots,v_{i_l}$ to $v_1,\ldots,v_l$.  To finally reach a contradiction, we prove the following.

\begin{claim}
\label{finishoff}
For\/ $l=2^{k^2+2k}$ and\/ $r=3\cdot2^k-4$, among the\/ $l$ fans\/ $A^{\ast}_1,\ldots,A^{\ast}_l$ with the properties above, there are\/ $k$ pairwise crossing curves.
\end{claim}

The proof follows the argument of Lemma 4.3 in \cite{FPS13}.  We proceed by induction on~$k$.  The base case $k=1$ is trivial.  For the inductive step, assume the statement holds up to $k-1$.  For simplicity, we let $a^{\ast}_{i,j}=a^{\ast}_{i,j'}$ for all $j\in\mathbb{Z}$, where $j'\in\{1,\ldots,r\}$ is such that $j\equiv j'\pmod{r}$.  Consider the fan $A^{\ast}_1$, which is of size~$r$.  By construction of $A^{\ast}_1$, the arrangement $A^{\ast}_1\cup \{\gamma^{\ast}\}$ partitions the plane into $r$ regions.  By the pigeonhole principle, there is a subset $V'\subset\{v_1,\ldots,v_l\}$ of size
\begin{equation*}
|V'|=\frac{l-1}{r}=\frac{2^{k^2+2k}-1}{3\cdot 2^k-4},
\end{equation*}
such that all the vertices in $V'$ lie in the same region.  Let $j_0\in\{1,\ldots,r\}$ be an integer such that $V'$ lies in the region bounded by $a^{\ast}_{1,j_0}$, $a^{\ast}_{1,j_0+1}$, and~$\gamma^{\ast}$.  See Figure~\ref{newbetween}.

\begin{figure}[t]
\begin{center}
\includegraphics[width=190pt]{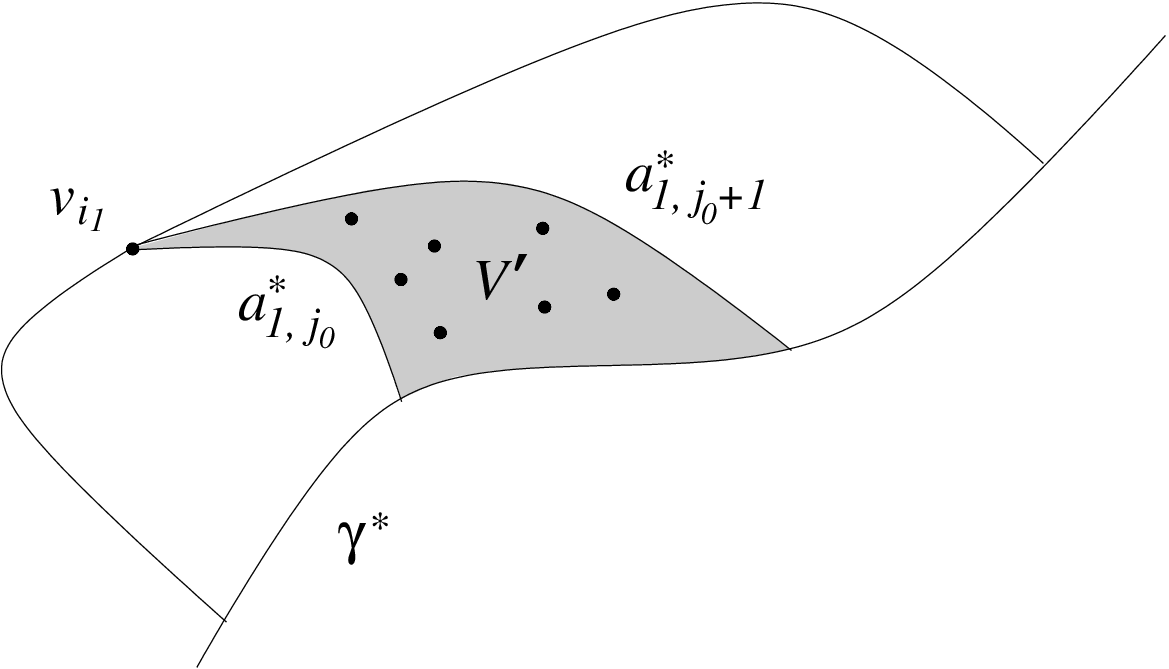}
\caption{Vertices of $V'$ lie in the region enclosed by $a^{\ast}_{1,j_0}$, $a^{\ast}_{1,j_0+1}$, and~$\gamma^{\ast}$.}
\label{newbetween}
\end{center}
\end{figure}

Let $v_i\in V'$ and $1<j_1<r$, and consider the curve $a^{\ast}_{i,j_0+j_1}$.  Recall that $a^{\ast}_{i,j_0+j_1}$ is disjoint from $\gamma^{\ast}_{j_0}\cup\gamma^{\ast}_{j_0+1}$ and thus intersects $a^{\ast}_{1,j_0}\cup a^{\ast}_{1,j_0+1}$.  Let $a\subset a^{\ast}_{i,j_0+j_1}$ be the maximal subcurve with an endpoint on $\gamma^{\ast}$ whose interior is disjoint from $a^{\ast}_{1,j_0}\cup a^{\ast}_{1,j_0+1}$.  If $a$ intersects $a^{\ast}_{1,j_0+1}$ (i.e.\ the second endpoint of $a$ lies on $a^{\ast}_{1,j_0+1}$), then $v_i$ and the left side of $\gamma^{\ast}_{j_0+2}\cup\cdots\cup\gamma^{\ast}_{j_0+j_1-1}$ lie in different connected components of $\mathbb{R}^2\setminus(a^{\ast}_{1,j_0+1}\cup\gamma^{\ast}\cup a)$.  Likewise, if $a$ intersects $a^{\ast}_{1,j_0}$, then $v_i$ and the left-side of $\gamma^{\ast}_{j_0+j_1+1}\cup\cdots\cup\gamma^{\ast}_{j_0+r-1}$ lie in different connected components of $\mathbb{R}^2\setminus(a^{\ast}_{1,j_0}\cup\gamma^{\ast}\cup a)$.

If $a$ intersects $a^{\ast}_{1,j_0+1}$, then all curves $a^{\ast}_{i,j_0+2},\ldots,a^{\ast}_{i,j_0+j_1-1}$ must also cross $a^{\ast}_{1,j_0+1}$.  Indeed, they connect $v_i$ with the left-side of $\gamma^{\ast}_{j_0+2}\cup\cdots\cup\gamma^{\ast}_{j_0+j_1-1}$, but their interiors are disjoint from $\gamma^{\ast}$ and $a^{\ast}_{i,j_0+j_1}$.  Likewise, if $a$ intersects $a^{\ast}_{1,j_0}$, then all curves $a^{\ast}_{i,j_0+j_1+1},\ldots,a^{\ast}_{i,j_0+r-1}$ must also cross $a^{\ast}_{1,j_0}$.  Therefore, we have the following.

\begin{claim}
\label{obs}
For half of the vertices\/ $v_i\in V'$, the curves emanating from\/ $v_i$ satisfy one of the following:
\begin{enumerate}
\item $a^{\ast}_{i,j_0+2},a^{\ast}_{i,j_0+3},\ldots,a^{\ast}_{i,j_0+r/2}$ all cross\/ $a^{\ast}_{1,j_0+1}$,
\item $a^{\ast}_{i,j_0+r/2+1},a^{\ast}_{i,j_0+r/2+2},\ldots,a^{\ast}_{i,j_0+r-1}$ all cross\/ $a^{\ast}_{1,j_0}$.
\end{enumerate}
\end{claim}

\noindent We keep all curves satisfying Claim \ref{obs}, and discard all other curves.  Since $r/2-2=3\cdot 2^{k-1}-4$ and
\begin{equation*}
\frac{|V'|}{2}\geq\frac{l-1}{2r}=\frac{2^{k^2+2k}-1}{6\cdot 2^k-8}\geq 2^{(k-1)^2+2(k-1)},
\end{equation*}
by Claim \ref{obs}, we can apply the induction hypothesis on these remaining curves which all cross $a^{\ast}_{1,j_0+1}$ or $a^{\ast}_{1,j_0}$.  Hence, we have found $k$ pairwise crossing edges, and this completes the proof of Claim \ref{finishoff} and thus Lemma~\ref{key}.
\end{proof}

Now, we are ready to prove Theorem~\ref{tint}.

\begin{proof}[Proof of Theorem \ref{tint}]
By Lemma \ref{regular} we know that, say, $S_1$ contains a $2l$-regular subsequence of length $\lceil|E'|/(8l)\rceil$.  By Theorem \ref{klazar} and Lemma \ref{key}, this subsequence has length at most
\begin{equation*}
n\cdot 2l\cdot 2^{(2lm-3)}\cdot(20l)^{10\alpha(n)^{2lm}}.
\end{equation*}
Therefore, we have
\begin{equation*}
\Bigl\lceil\frac{|E'|}{8l}\Bigr\rceil\leq n\cdot 2l\cdot 2^{(2lm-3)}\cdot(20l)^{10\alpha(n)^{2lm}},
\end{equation*}
which implies
\begin{equation*}
|E'|\leq 8n\cdot 2l^2\cdot 2^{(2lm-3)}\cdot(20l)^{10\alpha(n)^{2lm}}.
\end{equation*}
Since $l=2^{k^2+2k}$ and $m$ depends only on $k$ and $t$, for sufficiently large $c$ (depending only on $k$ and $t$) and $\alpha(n)\geq 2$, we have
\begin{equation*}
|E(G)|<2(k-1)^4|E'|+2n\leq 2^{\alpha(n)^c}n,
\end{equation*}
which completes the proof of Theorem~\ref{tint}.
\end{proof}

\section{Proof of Theorem \ref{simple}}\label{bart}

A family of curves in the plane is \emph{simple} if any two of them share at most one point.  A family $C$ of curves is \emph{$K_k$-free} if the intersection graph of $C$ is $K_k$-free, that is, no $k$ curves in $C$ pairwise intersect.  We let $\chi(C)$ denote the chromatic number of the intersection graph of $C$, that is, the minimum number of colors that suffice to color the curves in $C$ so that no two intersecting curves receive the same color.  A family $C$ of curves is \emph{pierced} by a line $\ell$ (a line segment $\beta$) if every curve in $C$ intersects $\ell$ ($\beta$) in exactly one point and this point is a proper crossing.

Our proof of Theorem \ref{simple} is based on the following result, proved in \cite{LMPW14} in a more general setting, for simple $K_k$-free families of compact arc-connected sets in the plane whose intersections with a line $\ell$ are non-empty segments.

\begin{theorem}[Laso\'n, Micek, Pawlik, Walczak \cite{LMPW14}]\label{thm:lason-et-al}
Every simple\/ $K_k$-free family of curves\/ $C$ pierced by a line\/ $\ell$ satisfies\/ $\chi(C)\leq a_k$, where\/ $a_k$ depends only on\/~$k$.
\end{theorem}

\noindent The dependence of $a_k$ on $k$ in Theorem \ref{thm:lason-et-al} is double exponential.  Special cases of Theorem \ref{thm:lason-et-al} have been proved by McGuinness \cite{McG00} for $k=3$ and by Suk \cite{Suk14} for $y$-monotone curves and any~$k$.  Recently, Rok and Walczak \cite{RoW14} extended Theorem \ref{thm:lason-et-al} to arbitrary (not necessarily simple) $K_k$-free families of curves pierced by a line $\ell$, but the corresponding constant $a_k$ in their theorem is enormous (an exponential tower of size~$k$).

The following is essentially a special case of a lemma due to McGuinness \cite[Lemma 2.1]{McG96}.  We include the proof for completeness.

\begin{lemma}[McGuinness \cite{McG96}]\label{lem:mcguinness}
Let\/ $G$ be a graph, $\prec$ be a total ordering of\/ $V(G)$, and\/ $c\geq 1$.  If\/ $\chi(G)>2c$, then\/ $G$ has an edge\/ $uv$ such that the subgraph of\/ $G$ induced on the vertices strictly between\/ $u$ and\/ $v$ in the order\/ $\prec$ has chromatic number at least\/~$c$.
\end{lemma}

\begin{proof}
Partition $V(G)$ into sets $V_1,\ldots,V_r$ that are pairwise disjoint intervals of the order $\prec$ so that $\chi(G[V_i])=c$ for $1\leq i<r$ and $\chi(G[V_r])\leq c$.  This can be done by adding vertices to $V_1$ from left to right in the order $\prec$ until we get $\chi(G[V_1])=c$, then following the same procedure with the remaining vertices to form $V_2$, and so on.  Color each $G[V_i]$ with $i$ odd properly with colors $\{1,\ldots,c\}$, and color each $G[V_i]$ with $i$ even properly with colors $\{c+1,\ldots,2c\}$.  If $\chi(G)>2c$, then the resulting $2c$-coloring of $G$ cannot be proper.  That is, $G$ has an edge $uv$ such that $u$ and $v$ are assigned the same color.  It follows that $u\in V_i$ and $v\in V_j$ for $i$ and $j$ distinct and of the same parity.  Therefore, at least one of the sets $V_k$ with $\chi(G[V_k])=c$ lies entirely between $u$ and $v$ in the order $\prec$, where $k$ is an index between $i$ and~$j$.
\end{proof}

Let $\ell$ be a horizontal line in the plane, and let $\beta$ be a segment of~$\ell$.  We will consider curves crossing $\beta$ at exactly one point, always assuming that this intersection point is distinct from the endpoints of~$\beta$.  Any such curve $\gamma$ is partitioned by $\beta$ into two subcurves: $\gamma^+$ that enters $\beta$ from above and $\gamma^-$ that enters $\beta$ from below, both including the intersection point of $\beta$ and~$\gamma$.

\begin{lemma}\label{lem:two-side}
Let\/ $C$ be a\/ $K_k$-free family of curves pierced by~$\beta$.  If\/ $\gamma_1^+\cap\gamma_2^+=\emptyset$ and\/ $\gamma_1^-\cap\gamma_2^-=\emptyset$ for any\/ $\gamma_1,\gamma_2\in C$, then\/ $\chi(C)\leq 2^{3k-6}$.
\end{lemma}

\begin{proof}
We proceed by induction on~$k$.  The base case $k=2$ is trivial, as a $K_2$-free family has chromatic number~$1$.  For the induction step, assume $k\geq 3$ and the statement holds up to $k-1$.  Assume for the sake of contradiction that $\chi(C)>2^{3k-6}$.  Let $\prec$ be the ordering of $C$ according to the left-to-right order of the intersection points with~$\beta$.  Apply Lemma \ref{lem:mcguinness} with $c=2^{3k-7}$.  It follows that there are two intersecting curves $\delta_1,\delta_2\in C$ such that $\chi(C(\delta_1,\delta_2))\geq 2^{3k-7}$, where $C(\delta_1,\delta_2)=\{\gamma\in C\colon\delta_1\prec\gamma\prec\delta_2\}$.  The curves $\beta$, $\delta_1$ and $\delta_2$ together partition the plane into two regions $R^+$ and $R^-$ so that for $\gamma\in C(\delta_1,\delta_2)$, $\gamma^+$ enters $\beta$ from the side of $R^+$, while $\gamma^-$ enters $\beta$ from the side of~$R^-$.  Take any $\gamma_1,\gamma_2\in C(\delta_1,\delta_2)$ that intersect at a point~$p$.  It follows from the assumptions of the lemma that $p\in\gamma_1^+\cap\gamma_2^-$ or $p\in\gamma_1^-\cap\gamma_2^+$.  If $p\in R^+$, then one of $\gamma_1^-$, $\gamma_2^-$ (whichever contains $p$) must intersect $\delta_1$ or~$\delta_2$.  Similarly, if $p\in R^-$, then one of $\gamma_1^+$, $\gamma_2^+$ must intersect $\delta_1$ or~$\delta_2$.  In both cases, one of $\gamma_1$, $\gamma_2$ intersects $\delta_1$ or~$\delta_2$.  Let $C_1$ and $C_2$ consist of those members of $C(\delta_1,\delta_2)$ that intersect $\delta_1$ and $\delta_2$, respectively.  Clearly, both $C_1$ and $C_2$ are $K_{k-1}$-free, and thus the induction hypothesis yields $\chi(C_1)\leq 2^{3k-9}$ and $\chi(C_2)\leq 2^{3k-9}$.  Moreover, we have $\chi\bigl(C(\delta_1,\delta_2)\setminus(C_1\cup C_2)\bigr)\leq 1$, as $C(\delta_1,\delta_2)\setminus(C_1\cup C_2)$ is independent by the assumption that $\gamma_1^+\cap\gamma_2^+=\emptyset$ and $\gamma_1^-\cap\gamma_2^-=\emptyset$ for any $\gamma_1,\gamma_2\in C$.  To conclude, we have $\chi(C(\delta_1,\delta_2))\leq 2\cdot 2^{3k-9}+1<2^{3k-7}$, which is a contradiction.
\end{proof}

\begin{theorem}\label{thm:curves}
Every simple\/ $K_k$-free family of curves\/ $C$ pierced by a line segment\/ $\beta$ satisfies\/ $\chi(C)\leq b_k$, where\/ $b_k$ depends only on\/~$k$.
\end{theorem}

\begin{proof}
As before, assume without loss of generality that $\beta$ is a segment of a horizontal line $\ell$ and no curve in $C$ passes through the endpoints of~$\beta$.  The family $C^+=\{\gamma^+\colon\gamma\in C\}$ can be transformed into a family $\tilde C^+=\{\tilde\gamma^+\colon\gamma\in C\}$ so that
\begin{itemize}
\item $\tilde C^+$ is simple,
\item each $\tilde\gamma^+$ is entirely contained in the upper half-plane delimited by~$\ell$,
\item $\tilde\gamma_1^+$ and $\tilde\gamma_2^+$ intersect if and only if $\gamma_1^+$ and $\gamma_2^+$ intersect.
\end{itemize}
This is achieved as follows.  Consider a closed Jordan curve $\beta^{\ast}$ such that $\beta\subset\beta^{\ast}$ and the interior of $\beta^{\ast}$ is disjoint from every $\gamma^+\in C^+$.  Such a curve exists by the definition of $C^+$ and the compactness of~$\beta$.  Then, invert the plane so that $\beta^{\ast}$ becomes $\ell$ and the exterior of $\beta^{\ast}$ becomes the upper half-plane delimited by~$\ell$.  In a similar way, the family $C^-=\{\gamma^-\colon\gamma\in C\}$ can be transformed into a family $\tilde C^-=\{\tilde\gamma^-\colon\gamma\in C\}$ so that
\begin{itemize}
\item $\tilde C^-$ is simple,
\item each $\tilde\gamma^-$ is entirely contained in the lower half-plane delimited by~$\ell$,
\item $\tilde\gamma_1^-$ and $\tilde\gamma_2^-$ intersect if and only if $\gamma_1^-$ and $\gamma_2^-$ intersect.
\end{itemize}
The curves $\tilde\gamma^+$ and $\tilde\gamma^-$ are respectively the upper and lower parts of the curve $\tilde\gamma=\tilde\gamma^+\cup\tilde\gamma^-$ intersecting $\ell$ at exactly one point.  The family $\tilde C=\{\tilde\gamma\colon\gamma\in C\}$ is clearly simple and $K_k$-free.  Therefore, by Theorem \ref{thm:lason-et-al}, $\chi(\tilde C)\leq a_k$.  Fix a proper $a_k$-coloring $\phi$ of $\tilde C$ and consider the set $C_i$ consisting of those $\gamma\in C$ for which $\phi(\tilde\gamma)=i$.  It follows that $\gamma_1^+\cap\gamma_2^+=\emptyset$ and $\gamma_1^-\cap\gamma_2^-=\emptyset$ for any $\gamma_1,\gamma_2\in C_i$.  Therefore, by Lemma \ref{lem:two-side}, $\chi(C_i)\leq 2^{3k-6}$.  Summing up over all colors used by $\phi$, we obtain $\chi(C)\leq 2^{3k-6}a_k$.
\end{proof}

The same proof but using the aforementioned extension of Theorem \ref{thm:lason-et-al} to arbitrary $K_k$-free families of curves pierced by a line $\ell$, due to Rok and Walczak \cite{RoW14}, yields an extension of Theorem \ref{thm:curves} to arbitrary (not necessarily simple) $K_k$-free families of curves pierced by a line segment~$\beta$.

Now, we are ready to prove Theorem~\ref{simple}.

\begin{proof}[Proof of Theorem \ref{simple}]
By Lemma \ref{reduction}, it is enough to prove that every $k$-quasi-planar simple topological graph on $n$ vertices that contains an edge intersecting every other edge has at most $c_kn$ edges, where $c_k$ depends only on~$k$.

Let $G$ be a $k$-quasi-planar simple topological graph on $n$ vertices, and let $pq$ be an edge that intersects every other edge.  Remove all edges with an endpoint at $p$ or $q$ except the edge $pq$.  Shorten each curve representing a remaining edge by a tiny bit at both endpoints, so that curves sharing an endpoint become disjoint, while all crossings are preserved.  The resulting set of curves $C$ is simple and $K_k$-free and contains a curve $\gamma$ crossing every other curve in~$C$.  Therefore, $C\setminus\{\gamma\}$ is $K_{k-1}$-free and $|C\setminus\{\gamma\}|>|E(G)|-2n$.  Since $C$ can be transformed into an equivalent set of curves so that $\gamma$ becomes the horizontal segment $\beta$, Theorem \ref{thm:curves} yields $\chi(C\setminus\{\gamma\})\leq b_{k-1}$.  Consequently, $C\setminus\{\gamma\}$ contains an independent set $S$ of size
\begin{equation*}
|S|\geq\frac{|C\setminus\{\gamma\}|}{b_{k-1}}>\frac{|E(G)|-2n}{b_{k-1}}.
\end{equation*}
The edges of $G$ corresponding to the curves in $S$ form a planar subgraph of $G$, which implies $|S|<3n$.  The two inequalities give $|E(G)|<(3b_{k-1}+2)n$.
\end{proof}

\linespread{1.08}
\bibliographystyle{plain}
\bibliography{quasi-planar}

\begin{thebibliography}{10}

\bibitem{Ack09}
Eyal Ackerman.
\newblock On the maximum number of edges in topological graphs with no four
  pairwise crossing edges.
\newblock {\em Discrete Comput. Geom.}, 41(3):365--375, 2009.

\bibitem{AAP+97}
Pankaj~K. Agarwal, Boris Aronov, J\'anos Pach, Richard Pollack, and Micha
  Sharir.
\newblock Quasi-planar graphs have a linear number of edges.
\newblock {\em Combinatorica}, 17(1):1--9, 1997.

\bibitem{AlS-book}
Noga Alon and Joel~H. Spencer.
\newblock {\em The Probabilistic Method}.
\newblock John Wiley \& Sons, Hoboken, 3rd edition, 2008.

\bibitem{BMP-book}
Peter Brass, William Moser, and J\'anos Pach.
\newblock {\em Research Problems in Discrete Geometry}.
\newblock Springer, New York, 2005.

\bibitem{CaP92}
Vasilis Capoyleas and J\'anos Pach.
\newblock A {T}ur\'an-type theorem on chords of a convex polygon.
\newblock {\em J. Combin. Theory Ser. B}, 56(1):9–--15, 1992.

\bibitem{Dil50}
Robert~P. Dilworth.
\newblock A decomposition theorem for partially ordered sets.
\newblock {\em Ann. Math.}, 51(1):161--166, 1950.

\bibitem{FoP12}
Jacob Fox and J\'anos Pach.
\newblock Coloring {$K_k$}-free intersection graphs of geometric objects in the
  plane.
\newblock {\em European J. Combin.}, 33(5):853--866, 2012.

\bibitem{FoP14}
Jacob Fox and J\'anos Pach.
\newblock Applications of a new separator theorem for string graphs.
\newblock {\em Combin. Prob. Comput.}, 23(1):66--74, 2014.

\bibitem{FPS13}
Jacob Fox, J\'anos Pach, and Andrew Suk.
\newblock The number of edges in {$k$}-quasi-planar graphs.
\newblock {\em SIAM J. Discrete Math.}, 27(1):550--561, 2013.

\bibitem{Kla92}
Martin Klazar.
\newblock A general upper bound in extremal theory of sequences.
\newblock {\em Comment. Math. Univ. Carolinae}, 33(4):737--746, 1992.

\bibitem{Kla02}
Martin Klazar.
\newblock Generalized {D}avenport-{S}chinzel sequences:\ results, problems, and
  applications.
\newblock {\em Integers}, 2:A11, 2002.

\bibitem{LMPW14}
Micha{\l} Laso\'n, Piotr Micek, Arkadiusz Pawlik, and Bartosz Walczak.
\newblock Coloring intersection graphs of arc-connected sets in the plane.
\newblock {\em Discrete Comput. Geom.}, 52(2):399--415, 2014.

\bibitem{McG96}
Sean McGuinness.
\newblock On bounding the chromatic number of {L}-graphs.
\newblock {\em Discrete Math.}, 154(1--3):179--187, 1996.

\bibitem{McG00}
Sean McGuinness.
\newblock Colouring arcwise connected sets in the plane {I}.
\newblock {\em Graph. Combin.}, 16(4):429--439, 2000.

\bibitem{Niv10}
Gabriel Nivasch.
\newblock Improved bounds and new techniques for {D}avenport-{S}chinzel
  sequences and their generalizations.
\newblock {\em J. Assoc. Comput. Machin.}, 57(3):1--44, 2010.

\bibitem{PRT06}
J\'anos Pach, Rado\v{s} Radoi\v{c}i\'c, and G\'eza T\'oth.
\newblock Relaxing planarity for topological graphs.
\newblock In Ervin Gy\H{o}ri, Gyula~O.H. Katona, and L\'aszl\'o Lov\'asz,
  editors, {\em More Graphs, Sets and Numbers}, volume~15 of {\em Bolyai Soc.
  Math. Stud.}, pages 285--300. Springer, Berlin, 2006.

\bibitem{PSS96}
J\'anos Pach, Farhad Shahrokhi, and Mario Szegedy.
\newblock Applications of the crossing number.
\newblock {\em Algorithmica}, 16(1):111--117, 1996.

\bibitem{Pet11a}
Seth Pettie.
\newblock Generalized {D}avenport-{S}chinzel sequences and their 0-1 matrix
  counterparts.
\newblock {\em J. Combin. Theory Ser. A}, 118(6):1863--1895, 2011.

\bibitem{Pet11b}
Seth Pettie.
\newblock On the structure and composition of forbidden sequences, with
  geometric applications.
\newblock In Ferran Hurtado and Marc~J. van Kreveld, editors, {\em 27th ACM
  Symposium on Computational Geometry (SoCG 2011)}, pages 370--379. ACM, New
  York, 2011.

\bibitem{RoW14}
Alexandre Rok and Bartosz Walczak.
\newblock Outerstring graphs are {$\chi$}-bounded.
\newblock In Siu-Wing Cheng and Olivier Devillers, editors, {\em 30th Annual
  Symposium on Computational Geometry (SoCG 2014)}, pages 136--143. ACM, New
  York, 2014.

\bibitem{Suk14}
Andrew Suk.
\newblock Coloring intersection graphs of {$x$}-monotone curves in the plane.
\newblock {\em Combinatorica}, 34(4):487--505, 2014.

\bibitem{Tur41}
P\'al Tur\'an.
\newblock Egy gr\'afelm\'eleti sz\'els{\H{o}}\'ert\'ek-feladatr\'ol ({O}n an
  extremal problem in graph theory).
\newblock {\em Mat. Fiz. Lapok}, 48:436--452, 1941.

\bibitem{Val98}
Pavel Valtr.
\newblock On geometric graphs with no {$k$} pairwise parallel edges.
\newblock {\em Discrete Comput. Geom.}, 19(3):461--469, 1998.

\end{thebibliography}

\end{document}